\documentclass[12pt]{article}
\usepackage{amssymb,amsthm,amsmath}
\usepackage{fullpage}
\usepackage{tikz}

\newcommand{\C}[1]{\left|#1\right|}

\newtheorem{theorem}{Theorem}[section]
\newtheorem*{thm1}{Theorem~\ref{thm:choice}}
\newtheorem{lemma}[theorem]{Lemma}

\newtheorem{proposition}[theorem]{Proposition}

\theoremstyle{definition}
\newtheorem{definition}[theorem]{Definition}

\def\NN{{\mathbb N}}
\def\FL#1{\left\lfloor{#1}\right\rfloor}
\def\FR{\frac}
\def\st{\colon\,}

\def\VEC#1#2#3{#1_{#2},\ldots,#1_{#3}}
\def\esub{\subseteq}
\def\FL#1{\left\lfloor{#1}\right\rfloor}

\begin{document}

\title{Coloring, sparseness, and girth}

\author{
Noga Alon\thanks{Sackler School of 
Mathematics and Blavatnik School of Computer
Science, Tel Aviv University, Tel Aviv 69978, 
Israel and School of Mathematics,
Institute for Advanced Study, Princeton, NJ 08540. Email: nogaa@tau.ac.il.
Research supported in part by a USA-Israeli BSF grant, 
by an ISF grant, by the
Israeli I-Core program and by the Oswald Veblen Fund.}\and
Alexandr Kostochka\thanks{Department of Mathematics, University of Illinois,
USA and Zhejiang Normal University, China, kostochk@math.uiuc.edu.  Research
supported in part by NSF grant    DMS-1266016.}\and
Benjamin Reiniger\thanks{Department of Mathematics, University of Illinois,
reinige1@illinois.edu.}\and
Douglas B. West\thanks{Departments of Mathematics, 
Zhejiang Normal University,
China, and University of Illinois, USA , west@math.uiuc.edu.  Research
supported by Recruitment Program of Foreign Experts, 1000 Talent Plan, State
Administration of Foreign Experts Affairs, China.}\and
Xuding Zhu\thanks{Department of Mathematics, Zhejiang Normal University,
xudingzhu@gmail.com.  Research supported by CNSF 11171310.}
}
\date{\today}
\maketitle

\begin{abstract}
An \emph{$r$-augmented tree} is a rooted tree plus $r$ edges added from each
leaf to ancestors.  For $d,g,r\in\mathbb{N}$, we construct a bipartite
$r$-augmented complete 
$d$-ary tree having girth at least $g$.  The height of
such trees must grow extremely rapidly in terms of the girth.

Using the resulting graphs, we construct sparse non-$k$-choosable bipartite
graphs, showing that maximum average degree at most $2(k-1)$ is a sharp
sufficient condition for $k$-choosability in bipartite graphs, even when
requiring large girth.  We also give a new simple construction of 
non-$k$-colorable
graphs and hypergraphs with any girth $g$.
\end{abstract}

\baselineskip=16pt

\section{Introduction}

A graph $G$ is \emph{$k$-choosable} if, for every way of assigning a list $L(v)$
of $k$ colors to each vertex $v\in V(G)$, there is a proper coloring $f$ of
$G$ with $f(v)\in L(v)$ for all $v$.  The \emph{choice number} of a graph
is the least $k$ such that it is $k$-choosable.  If every subgraph has average
degree less than $k$, then it has a vertex with degree less than $k$, and
inductively it is $k$-choosable.

For bipartite graphs, one can guarantee $k$-choosability with average degree
up to $2(k-1)$.  Using (an early version of) 
the Combinatorial Nullstellensatz \cite{Al}, Alon and
Tarsi~\cite{AT} proved Theorem~\ref{choose} below, 
which implied the conjecture
of~\cite{ERT} that planar bipartite graphs are $3$-choosable.  As mentioned
in~\cite{AT}, another route to the result was subsequently 
noted by Bondy, Boppana, and Siegel, as follows.
A \emph{kernel} of a digraph is an independent set $S$ containing a
successor of every vertex outside $S$.  
If a graph $G$ has an orientation $D$
with maximum outdegree less than $k$, 
and every induced subdigraph of $D$ has a
kernel, then inductively $G$ is $k$-choosable.  Richardson~\cite{R} proved that
every digraph with no odd cycle has a kernel.  Hakimi~\cite{H} proved that 
$G$ has an orientation with maximum outdegree at most $k-1$ when all induced
subgraphs have average degree at most $2(k-1)$.

\begin{theorem}[\cite{AT}]\label{choose}
If $G$ is a bipartite graph such that every subgraph has average degree at
most $2(k-1)$, then $G$ is $k$-choosable.
\end{theorem}

We show that Theorem~\ref{choose} is sharp in a strong sense: we construct
non-$k$-choosable bipartite graphs $G$ such that after deleting any edge from 
$G$, all subgraphs of the remaining graph have average degree at most $2(k-1)$.
Thus our graphs are $(k+1)$-choice-critical.
Furthermore, such examples exist with arbitrarily large girth.  We prove the
following theorem.

\begin{theorem}\label{thm:choice}
For $g,k\in\mathbb{N}$, there is a bipartite graph $G$ with girth at least $g$
that is not $k$-choosable even though every proper subgraph has average degree
at most $2(k-1)$.
\end{theorem}

To prove this, we consider a new problem.  Let an \emph{$r$-augmented tree} be
a graph consisting of a rooted tree (called the \emph{underlying tree}) plus
edges from each leaf to $r$ of its ancestors (called \emph{augmenting edges}).
A \emph{complete $d$-ary tree of height $m$} is a rooted tree whose internal
vertices have $d$ children and whose leaves have distance $m$ from the root.
For $d,r,g\in\NN$, let a \emph{$(d,r,g)$-graph} be a bipartite $r$-augmented
complete $d$-ary tree with girth at least $g$.

\begin{theorem}\label{thm:raugtrees}
For $d,r,g\in\mathbb{N}$, there exists a $(d,r,g)$-graph.
\end{theorem}

In Section~\ref{sec:raugtrees} we prove Theorem~\ref{thm:raugtrees}, and 
in Section~\ref{sec:appl} we give several applications.  
In Section~\ref{sec:chromatic} we present a simple construction of $t$-uniform
hypergraphs with arbitrarily large girth and chromatic number, for all $t$.
For $t=2$, Erd\H{o}s~\cite{Er} used the probabilistic method 
to prove existence; see also \cite{EH,KN} for subsequent work.
Explicit constructions followed in~\cite{Kr,Lo,NR}. These are
inductive and, except for \cite{Kr},
use hypergraphs with large edges.  Using $(d,r,g)$-graphs (built
inductively), our construction is non-inductive and does not involve
hypergraphs with larger edges. Moreover, the same method
provides explicit high girth hypergraphs of any uniformity 
based on $(d,r,g)$-graphs, without using
hypergraphs (besides those constructed) in the process. 

We prove Theorem~\ref{thm:choice} in Section~\ref{sec:choice}.  Stronger
versions involving restricted list assignments are proved in
Section~\ref{sec:variations}.  For example, when the lists at adjacent vertices
are disjoint, every coloring chosen from the lists is proper.  We extend the
analysis of the graph constructed for Theorem~\ref{thm:choice} by constructing
a $k$-list assignment in which any two adjacent lists have exactly one common
color and yet no proper coloring can be chosen.

\nobreak
One can also restrict list assignments by bounding the size of the union of the
lists.  For bipartite graphs, a proper coloring can be chosen from any
$k$-lists whose union has size at most $2k-2$.  We prove that this is sharp
(for any girth) by constructing a bipartite graph with $k$-lists whose union
has size $2k-1$ from which no proper coloring can be chosen.

Finally, in Section~\ref{sec:height} we discuss the height of 
the trees used in
Theorem~\ref{thm:raugtrees}.  For fixed $d \geq 2$ and $r \geq 1$, 
we show that the height
must grow extremely rapidly in terms of the girth.


\section{Augmented trees}\label{sec:raugtrees}

In this section and throughout, we restrict $g$ to be even.  If there is a
$(d,r,g)$-graph, then let $m(d,r,g)$ denote the least height of the underlying
tree in such a graph (otherwise, let $m(d,r,g)=\infty$).
Theorem~\ref{thm:raugtrees} is the statement that
$m(d,r,g)$ is finite for all $d,r,g\in\NN$.  We prove this by double induction,
using the following three lemmas.
\begin{lemma}\label{lem:11}
For $d,r\in\NN$, we have $m(d,r,4) = 2r+1$.
\end{lemma}
\begin{lemma}\label{lem:12}
For $g,d\in\NN$ with $g$ at least $4$ and even,
$m(d,1,g+2) \leq 2+m(d,d^2,g)$.
\end{lemma}
\begin{lemma}\label{lem:13}
With $d,r,g$ as above, $m(d,r+1,g) \le m_1+m_2-1$, where
$m_1=2\FL{\FR{m(d,1,g)}2}+1$ and $m_2=m(d^{m_1},r,g)$.
\end{lemma}
These three lemmas imply the finiteness of $m(d,r,g)$ for all $d,r,g\in\NN$
with $g$ even and at least $4$.  Letting $P(r,g)$ denote the claim that
$m(d,r,g)$ is finite for all $d$, we prove $P(r,g)$ by induction on $g$. 
As the base step, $P(r,4)$ holds for all $r$ by Lemma~\ref{lem:11}.  If
$P(r,g)$ holds for all $r$, then we prove $P(r,g+2)$ by induction on $r$:
first $P(1,g+2)$ holds by Lemma~\ref{lem:12} (using the truth of $P(r,g)$ for
all $r$), and then $P(r+1,g+2)$ follows from $P(r,g+2)$ by Lemma~\ref{lem:13}
(since $P(1,g+2)$ also holds).  This completes the proof of
Theorem~\ref{thm:raugtrees}.

It remains to prove the three lemmas. Lemma \ref{lem:11} is trivial: just make
each leaf adjacent to its $r$ non-parent ancestors at odd distance from it in
the tree.

\begin{proof}[Proof of Lemma~\ref{lem:12}]
Let $G'$ with underlying tree $T'$ be a $(d,d^2,g)$-graph with height
$m(d,d^2,g)$.  Replace each leaf $v$ of $T'$ with a complete $d$-ary tree $T_v$
of height $2$ rooted at $v$.  Replace the augmenting edges from $v$ to its
ancestors by letting the $d^2$ lower endpoints be the leaves of $T_v$ instead
of $v$.  This produces a $1$-augmented complete $d$-ary tree $G$ of height
$2+m(d,d^2,g)$.  Since each augmenting edge has had its lower endpoint moved 
two levels down, $G$ is bipartite.

If $G$ has a cycle $C$ of length at most $g$, then $C$ must contain an
augmenting edge, say $xy$, with $y$ being a leaf in the underlying tree $T$ of
$G$.  Let $v$ be the leaf in $T'$ such that $y$ is in $T_v$.  Since $d_G(y)=2$,
the cycle $C$ contains the edge $yy'$ of $T_v$ incident with $y$.  Contracting
the added subtrees of height $2$ into leaves of $T'$ contracts $C$ to a closed
walk $C'$ in $G'$ of length less than $g$. Since $C'$ traverses edge $vx$ only
once, the remaining walk from $x$ to $v$ along $C'$ contains a path that with
$vx$ completes a cycle of $G'$ having length less than $g$, a contradiction.
Thus $G$ has no cycles of length less than $g+2$.
\end{proof}


\begin{proof}[Proof of Lemma \ref{lem:13}]
Fix $r$.  Assuming for all $d$ and $g$ that $m(d,r,g)$ and $m(d,1,g)$ are
finite, let $m_1=2\FL{m(d,1,g)/2}+1$ and $m_2=m(d^{m_1},r,g)$.  Note that
$m_1$ is the least odd integer that is at least $m(d,1,g)$.  We construct the
desired graph $G$ from two graphs $G_1$ and $G_2$.

For $G_1$ we use a $(d,1,g)$-graph having height $m_1$.  If $m(d,1,g)$ is odd,
then $m_1=m(d,1,g)$ and we use a shortest $(d,1,g)$-graph.  If $m(d,1,g)$ is
even, then $m_1=m(d,1,g)+1$, and we form $G_1$ from $d$ copies of a shortest
$(d,1,g)$-graph by adding a new root having the roots of those graphs as
children.

For $G_2$, let $d'=d^{m_1}$, and consider a $(d',r,g)$-graph $H$ having height
$m_2$.  Let $G_2$ be an induced subgraph of $H$ formed by starting from the
root of the underlying tree of $H$ and keeping only $d$ children of each
included vertex, except that all $d'$ children are kept at the last level.
Thus $G_2$ has an underlying tree $T'$ of height $m_2$, and deleting the
$d^{m_2-1}d'$ leaves of $T'$ yields a complete $d$-ary tree of height $m_2-1$.
All ancestors in $H$ of a leaf of $T'$ appear in $T'$, so each leaf of $T'$
has $r$ ancestors as neighbors in $G_2$.

Now we construct $G$ from $G_1$ and $G_2$.  In $G_2$, let $S(u)$ be the star
consisting of a vertex $u$ at level $m_2-1$ and its $d'$ leaf children.
Replace each $S(u)$ with a copy $G_1(u)$ of the graph $G_1$, so that the $d'$
leaves in $G_1$ each become one of the leaves in $S(u)$, inheriting the $r$
augmenting edges that were incident to that leaf in $G_2$.  We call the
augmenting edges obtained from $G_2$ in this way \emph{long edges}; the
augmenting edges in $G_1(u)$ are \emph{short edges}.

The underlying tree in our construction thus has two parts.  The top part is
the tree $T'$ for $G_2$ without its bottom level; it has height $m_2-1$.  The
bottom part, with height $m_1$, consists of copies of $G_1$.  Each leaf has one
incident short edge from $G_1$ and $r$ incident long edges inherited from $G_2$.
Thus $G$ is an $(r+1)$-augmented complete $d$-ary tree of height $m_1+m_2-1$.
When replacing one of the $r$ augmenting edges from a leaf of $G_2$ by a long
edge, the difference in the heights of the endpoints increases by $m_1-1$.
Since $m_1$ is odd, this change is even, so $G$ is bipartite.

A cycle $C$ in $G$ that contains no long edges is a cycle in a copy of $G_1$ and
hence has length at least $g$.  When $C$ contains a long edge, contracting a
subtree $G_1(u)$ into a star $S(u)$ contracts $C$ to a closed walk $C'$ in
$G_2$ using an augmenting edge $e$.  Since leaves of $G_1(u)$ correspond
bijectively to leaves of $S(u)$, the edge $e$ is not repeated in $C'$.  Hence
the other walk in $C'$ joining its endpoints contains a path that completes a
cycle with $e$.  Since this is a cycle in $G_2$ and has length at least $g$,
also $C$ has length at least $g$.

This completes the proof of Lemma~\ref{lem:13} and Theorem~\ref{thm:raugtrees}.
\end{proof}

\section{Applications}\label{sec:appl}

In a complete $k$-ary tree, a \emph{full path} is a path from the root to a
leaf.  Let $[k]=\{1,\dots,k\}$.  A \emph{$[k]$-coloring} is a $k$-coloring 
using the colors in $[k]$.

\begin{definition}\label{fullpath}
Given an ordering of the children at each internal vertex, the vertices of a
complete $k$-ary tree with height $m$ correspond naturally to the strings
of length at most $m$ from the alphabet $[k]$.  Define an edge-coloring $\phi$
by letting the color of each edge from parent $x$ to child $y$ be the index of
$y$ in the ordering of the children of $x$ (note that $\phi$ is not a proper
coloring).  For a $[k]$-coloring $f$ of the vertices of $T$, a full path $P$ is an
\emph{$f$-path} if the color of each non-leaf vertex on $P$ equals the color
of the edge to its child on $P$.
\end{definition}

Whenever $f$ is a $[k]$-coloring of a complete $k$-ary tree, there is a unique
$f$-path: just start from the root and repeatedly follow the descending edge
whose color matches the color of the current vertex.  Similarly, every full
path is an $f$-path for some $[k]$-coloring $f$.

\subsection{Large chromatic number and girth}\label{sec:chromatic}

As mentioned in the introduction, there exist $t$-uniform hypergraphs with
large chromatic number and girth.  Our $(d,r,g)$-graphs provide a remarkably
simple such construction.  It has the benefits of
being non-recursive (once $(d,r,g)$-graphs are constructed), 
and not involving
hypergraphs as inputs to the construction. Thus unlike the earlier
constructions which use hypergraphs to provide high girth graphs,
the method described here constructs high girth graphs and hypergraphs 
using only graphs.

\begin{theorem}[{\rm\cite{Er,EH,Lo,NR,Kr,KN}}]\label{thm:chromatic1}
For $k,g,t\in\NN$, there is a $t$-uniform hypergraph 
with girth at least $g$
and chromatic number larger than $k$.
\end{theorem}
\begin{proof}
Let $G$ be a $(k,(t-1)k+1,2g)$-graph with underlying tree $T$ having leaf set
$L$.  Let $V'=V(T)-L$.  For $v\in L$, consider the full path $P$ ending at $v$.
Among the $(t-1)k+1$ neighbors of $v$ via augmenting edges, the pigeonhole
principle yields a set of $t$ neighbors of $v$ whose descending edges along $P$
have the same color; let $e_v$ be such a set of vertices in $V'$.  Let $H$ be
the $t$-uniform hypergraph with vertex set $V'$ and edge set
$\{e_v\st v\in L\}$.

Any $[k]$-coloring $f$ of $V'$ yields a unique $f$-path in $T$, ending at
some leaf $v$.  As a coloring of $H$, this makes the edge $e_v$ monochromatic.
Hence $H$ has no proper $k$-coloring.

Let $C$ be a shortest cycle in $H$, with edges $\VEC e1l$ in order and vertex
$x_i$ chosen from $e_{i-1}\cap e_i$ (subscripts modulo $l$).  Since $C$ is a
shortest cycle, $\VEC x1l$ are distinct.  Each edge of $H$ consists of
neighbors of a single leaf of $T$ via augmenting edges; let $v_i$ be the common
leaf neighbor of the vertices in $e_i$.  Form $C'$ in $G$ by replacing each edge
$e_i$ of $C$ by the copy of $P_3$ in $G$ having endpoints $x_{i-1}$ and $x_{i}$
and midpoint $v_i$.  Since for each leaf of $T$ we formed exactly one edge in
$H$, the leaves $\VEC v1l$ are distinct.  Hence $C'$ is a cycle, and its length
is twice that of $C$.  By the choice of $G$ as a $(d,k,2g)$-graph, $H$ has
girth at least $g$.
\end{proof}

The hypergraph $H$ in Theorem~\ref{thm:chromatic1} satisfies $|E(H)|=|L|=k^h$
and $|V(H)|=|V'|=\frac{k^h-1}{k-1}$, where $h=m(k,(t-1)k+1,2g)$.  Hence
$|E(H)|=(k-1)|V(H)|+1$.  However, $H$ may have (and actually does
have) dense subgraphs.
For $t=2$, we provide a different construction, inductive, of sparse graphs
with large girth and chromatic number.  A graph $G$ is sparse when it has a
small value of the \emph{maximum average degree}, defined to be
$\max_{H\esub G}\FR{\sum_{v\in V(H)}d_H(v)}{|V(H)|}$.  
Our construction has asymptotically lowest average degree even in the
broader class of triangle-free graphs. This follows from the lower bound
by Kostochka and Stiebitz~\cite{KS}: every $k$-chromatic
 triangle-free graph has maximum average degree at least
 $2k-o(k)$.

\begin{definition}\label{reduced}
Let $G$ be a $(d,r,g)$-graph with a specified ordering of the $d$ children at
each non-leaf vertex of the underlying tree $T$.  The corresponding
\emph{reduced $(d,r,g)$-graph} $H$ is obtained from $G$ as follows: given the
coloring $\phi$ of $E(G)$ from Definition~\ref{fullpath}, form $H$ from $G$ by
deleting at each non-root internal vertex $v$ of $T$ the subtree under the
descending edge whose color under $\phi$ is the same as the color of the edge
to the parent of $v$.  Each non-leaf vertex of $H\cap T$ has degree $d$ in $T$,
and $\phi$ is a proper edge-coloring of $H\cap T$.
\end{definition}

The reduced $(d,r,g)$-graph with underlying tree $T$ associated with the
edge-coloring $\phi$ as in Definition~\ref{reduced} still has a unique
$f$-path for any proper $[d]$-coloring $f$ of $T$.

\begin{theorem}\label{sparse}
For $k,g\in\NN$, there is a graph with girth at least $g$ that is 
not $k$-colorable and has maximum average degree at most $2(k-1)$.
\end{theorem}
\begin{proof}
For fixed $g$, we construct such a graph $J_k$ by induction on $k$.
For the basis step, let $J_2$ be an odd cycle of length at least $g$.
Given $J_{k-1}$, let $r=|V(J_{k-1})|$.

Let $H$ be a reduced $(k,(r-1)k+1,g)$-graph, with underlying tree $T$ and
edge-coloring $\phi$.  For each leaf $v$ of $T$, consider the full path $P$
ending at $v$.  By the pigeonhole principle, 
some $r$ neighbors of $v$ in $H$
(via augmenting edges) have the same color on their descending edges along $P$.
Keep the augmenting edges from $v$ to one such set and delete the other
augmenting edges.  The resulting graph $H'$ is a reduced $(k,r,g)$-graph.

Next replace each leaf $v$ of $H'$ with a copy of $J_{k-1}$; each vertex in the
copy for $v$ inherits exactly one augmenting edge of $H'$ from $v$.  This is
the graph $J_k$.  The edge to $v$ in $T$ disappears; vertices at the level just
before the leaves no longer have edges to children.

Any proper $[k]$-coloring $f$ of $V(T)$ yields a unique $f$-path; it ends at
some leaf $v$.  Because it is an $f$-path, the colors on the vertices match the
colors on the descending edges.  Let $Q$ be the copy of $J_{k-1}$
corresponding to $v$ in $J_k$.  By the construction of $J_k$, there is a fixed
color $c$ that appears on the neighbor in $V(T)$ of each vertex in $Q$.  Since
$J_{k-1}$ is not $(k-1)$-colorable, we cannot complete a proper $k$-coloring of
$J_k$. 

A cycle in one copy of $J_{k-1}$ has length at least $g$.  For any other cycle
$C$ in $J_k$, contracting each copy of $J_{k-1}$ to a single vertex yields a
closed walk $C'$ in $H'$ using some augmenting edge.  Since each vertex in
a copy of $J_{k-1}$ inherits only one augmenting edge, each augmenting edge
is used only once in $C'$.  Hence as in the proof of Lemma~\ref{lem:13}, $C'$ contains
a cycle in $H'$.  This cycle has length at least $g$, so $C$ has length at
least $g$.

For the maximum average degree, consider a subgraph $F$, and let $F'=F-V(T)$.
Being contained in copies of $J_{k-1}$, the graph $F'$ has average degree at
most $2(k-2)$.  Augmenting edges add at most $1$ to the degree of each vertex
of $F'$ and hence at most $2$ to the degree-sum in $F$ for each vertex in $F'$.
Working upward in $T$, each added vertex in $F$ adds at most $k-1$ downward
edges, which contributes at most $2(k-1)$ to the degree-sum.  The root may add
$k$ downward edges, but the lowest vertex added from $T$ adds fewer than $k-1$.
Thus the degree-sum is at most $2(k-1)$ per vertex of $F$.
\end{proof}

\subsection{Choosability}\label{sec:choice}

A modification of the construction in Theorem~\ref{sparse} yields 
non-$k$-choosable bipartite graphs that are as sparse as can be.  As noted in
Theorem~\ref{choose}, every bipartite graph with maximum average degree at most
$2(k-1)$ is $k$-choosable.  Hence the graphs we construct in
Theorem~\ref{thm:choice} with just one extra edge are $(k+1)$-choice-critical.

It is well known (since~\cite{ERT}) that a bipartite graph consisting of two
even cycles sharing one vertex is not $2$-choosable; indeed, it is
$3$-choice-critical.

\begin{thm1}
For $k\ge2$ and $g\ge4$, there is a bipartite graph $G_k$ with girth at least
$g$ that is not $k$-choosable even though every proper subgraph has average
degree at most $2(k-1)$.
\end{thm1}

\begin{proof}
We proceed by induction on $k$ for even $g$.  To count edges in subgraphs, we
will orient $G_k$ and count edges by their tails.  The orientation gives each
vertex outdegree $k-1$ except a designated root vertex, which has outdegree
$k$, and every vertex will be reachable from the root.  Thus $G_k$ will have
$(k-1)\C{V(G_k)}+1$ edges, and every proper subgraph will have smaller
outdegree at some vertex and thus have average degree at most $2(k-1)$.

Let $G_2$ be the graph consisting of two $g$-cycles sharing one vertex, which
is the root.  Orient $G_2$ consistently along each of the two cycles.
The desired properties hold.

For $k\geq3$, suppose that $G_{k-1}$ has all the desired properties.
Let $r=|V(G_{k-1})|-1$, and let $H'$ be a reduced $(k,r,2g)$-graph, with
underlying tree $T$.  We modify the bipartite graph $H'$ slightly to guarantee
that $G_k$ will be bipartite.  Let $(A,B)$ be the bipartition of $G_{k-1}$,
with $A$ containing the root, and let $a=|A|-1$ and $b=|B|$.  Each leaf $v$ in
$T$ has $a+b$ incident augmenting edges.  Let $A(v)$ denote some set of $a$ of
these edges.  For the remaining $b$ augmenting edges incident to $v$, move
their endpoints in the tree one step closer to $v$ along the full path to $v$.
Let $B(v)$ denote this new set of $b$ augmenting edges at $v$.  Let $H$ be the
resulting graph; $H$ is a reduced $(k,r,g)$-graph except for not being
bipartite.

Form $G_k$ from $H$ by adding a copy of $G_{k-1}$ for each leaf $v$ of $T$,
merging $v$ with the root of $G_{k-1}$, with each vertex of $A$ in the copy of
$G_{k-1}$ (other than the root) inheriting one edge of $A(v)$ and each vertex
of $B$ in the copy of $G_{k-1}$ inheriting one edge of $B(v)$.  Since the 
vertices of $B$ have odd distance from $v$ in $G_{k-1}$, this guarantees that
$G_k$ is bipartite.

Designate the root of $T$ as the root of $G_k$.  Orient the edges of $T$ away
from the root, keep the orientation guaranteed by the induction hypothesis
on the copies of $G_{k-1}$, and orient the augmenting edges away from the
copies of $G_{k-1}$.  Because $H'$ is a reduced $(k,r,2g)$-graph, every vertex
has outdegree $k-1$ except that the root has outdegree $k$.

Let $L'$ be an assignment of lists of size $k-1$ to $G_{k-1}$ such that
$G_{k-1}$ is not $L'$-colorable and none of these lists intersects $[k]$.  
Form a list assignment $L$ for $G_k$ as
follows.  Put $L(x)=[k]$ for each non-leaf vertex $x$ in $V(T)$.  For each leaf
$v\in V(T)$ and each vertex $w$ of $V(G_{k-1})$, let $w_v$ denote the copy of
$w$ in the copy of $G_{k-1}$ at $v$.  Let $P$ be the full path in $T$ ending at
$v$.  Let $L(w_v)=L'(w)\cup\{c\}$, where $c$ is the color on the edge of $P$
descending from the neighbor of $w_v$ in $V(P)$.  In particular, when $w$ is
the root, the added color is the color on the edge of $T$ reaching $v$.

Let $f$ be a coloring of $G_k$ with $f(u)\in L(u)$ for $u\in V(G_k)$.  If $f$
is proper on $T$, then since $f(x)\in[k]$ for $x\in V(T)$, there is a unique
$f$-path $P$ in $T$.  In the copy of $G_{k-1}$ for the leaf $v$ at the end of
$P$, the color $c$ that was added to each list is now forbidden in a proper
coloring, leaving the list $L'(w)$ at $w_v$.  By the choice of $L'$, a proper
coloring cannot be completed from these lists.
\end{proof}


\subsection{Restricted list colorings}\label{sec:variations}
As described in the introduction, we now strengthen Theorem~\ref{thm:choice}
by proving non-choosability results for restricted list assignments.  We
consider both restrictions on the intersections of adjacent lists and 
restrictions on the size of the union of the lists.

Every graph is $L$-colorable (by choosing arbitrarily) when adjacent vertices
have disjoint lists, but $L$-colorability may fail when adjacent lists are
almost disjoint.  List coloring with intersection constraints on adjacent lists
has been studied by Kratochv\'il, Tuza, and Voigt \cite{KTV} and by F\"uredi,
Kostochka, and Kumbhat~\cite{FKK}.  We next strengthen Theorem~\ref{thm:choice}
by showing that our graph $G_k$ fails to be $L$-colorable for a
particular $k$-list assignment $L$ such that $|L(u)\cap L(v)|=1$ for every edge
$uv$.

\begin{theorem}\label{listcap}
Fix $g\in\NN$ with $g\equiv4\pmod6$.  For $k\ge2$, the bipartite graph $G_k$
with girth at least $g$ constructed in Theorem~\ref{thm:choice} admits a
$k$-list assignment $L$ such that $G_k$ is not $L$-colorable despite
satisfying $|L(u)\cap L(v)|=1$ for all $uv\in E(G_k)$.
\end{theorem}
\begin{proof}
For $k=2$, let $u$ be the common vertex of the two cycles in $G_2$.
Set $L(u)=\{1,2\}$.  On each of the two cycles, the number of remaining
vertices is a multiple of $3$.  Along one cycle, rotate through the lists
$\{1,3\},\{3,4\},\{4,1\}$.  This forces color $1$ onto a neighbor of $u$.
On the other cycle substitute $2$ for $1$, forcing color $2$ onto a neighbor
of $u$.  Now $u$ cannot be colored.  Adjacent lists share one color.
 
For $k\ge3$, let $T$ be the underlying tree in $G_k$.  Color the edges of
$T$ by distinct colors.  For a non-leaf vertex $x$ in $T$, let $L(x)$ be the
set of colors on the edges incident to $x$; thus lists adjacent via edges of
$T$ have one common color.
 
By the induction hypothesis, there is a $(k-1)$-list assignment $L'$ on
$G_{k-1}$ such that $G_{k-1}$ is not $L'$-colorable.  For each leaf $v\in V(T)$,
let $L'_v$ be a copy of this assignment indexing the colors by $v$, so that the
colors used for the copy $G'$ of $G_{k-1}$ at $v$ will not be used anywhere
else.  For each vertex $w$ of $V(G_{k-1})$ other than the root, let $w_v$ denote
the copy of $w$ in $G'$.  Let $P$ be the full path in $T$ ending at $v$.  Let
$x$ be the neighbor of $w_v$ in $V(P)$, and let $c_x$ be the color of the edge
in $P$ descending from $x$ along $P$.  Let $L(w_v)=L'_v(w)\cup\{c_x\}$.
Let $L(v)=L'_v(v)\cup\{c_v\}$, where $c_v$ is the color of the edge incident
to $v$ in $T$.

For any proper coloring $f$ of $T$ chosen from these lists, there is a unique
full path $Q$ such that the color of each non-leaf vertex is the color of the
edge to its child on $Q$, constructed from the root: that is, an $f$-path.
Let $v$ be the leaf reached by $Q$.  The parent of $v$ has been given color
$c_v$, so that color cannot be used at $v$.  Similarly, for each other vertex
in the copy of $G_{k-1}$ at $v$, the added color in its list has been used
on its neighbor in $T$.  Finding an $L$-coloring of $G_k$ thus requires finding
an $L'$-coloring of $G_{k-1}$, which does not exist.
\end{proof}

Perhaps surprisingly, for bipartite graphs larger intersections than in
Theorem~\ref{listcap} also guarantee $L$-colorability, giving 
the sharpness of
Theorem~\ref{listcap} in another way.

\begin{proposition}\label{twocommon}
If $G$ is a bipartite graph, and $L$ is a list assignment such that any
two adjacent lists have at least two common elements (the lists may have
any sizes at least $2$), then $G$ is $L$-colorable.
\end{proposition}
\begin{proof}
Let $X$ and $Y$ be the parts of $G$, and index the colors in
$\bigcup_{v\in V(G)}L(v)$ as $c_1,\dots,c_t$.  Color each vertex of $X$ with
the highest-indexed color in its list and each vertex of $Y$ with the
lowest-indexed color in its list.  If two adjacent vertices receive the same
color, then it is the only common color in their lists, a contradiction.  Hence
the coloring is proper.
\end{proof}

When $G$ is $j$-colorable but not $k$-choosable, one may ask how large the
union $U$ of the lists must be in a $k$-list assignment $L$ such
that $G$ is not $L$-colorable.  Trivially $|U|>j$ is needed.  In fact,
one needs somewhat more, which reduces to $2k-1$ when $j=2$.

\begin{proposition}\label{smallcup}
Let $G$ be a $j$-colorable graph, with $j\le k$.  If $L$ is a $k$-list
assignment on $G$ such that $|\bigcup_{v\in V(G)}L(v)|\le\FR{j(k-1)}{j-1}$,
then $G$ is $L$-colorable.  Furthermore, the bound is sharp.
\end{proposition}
\begin{proof}
Let $f$ be a proper $j$-coloring of $G$.  Let $U=\bigcup_{v\in V(G)}L(v)$.
Split $U$ into disjoint sets $U_1,\dots,U_j$, with the smallest having size
$\FL{|U|/j}$.  Since $|U|\le\FR{j(k-1)}{j-1}$, the largest $j-1$ of the sets
together have size at most $k-1$.  (Note that
$\FL{\frac{j(k-1)}{j-1}}-\FL{\frac{k-1}{j-1}}=k-1$, and when
$|U|<\FL{\frac{j(k-1)}{j-1}}$ the conclusion becomes easier.)
Thus each $k$-list $L(v)$ intersects each $U_i$.  Hence each vertex $v$ can
choose a color from $L(v)\cap U_{f(v)}$.
Such a coloring is proper.

For sharpness, consider a universe $U$ of colors, and let $G$ be a complete
$j$-partite graph with $\binom{|U|}k$ vertices in each part.  Assign lists by
letting $L$ give each $k$-subset of $U$ as a list to one vertex in each part.
In an $L$-coloring, each color can be chosen in only one part.  Since a color
must be chosen from every vertex, on each part at least $|U|-(k-1)$ colors must
be chosen.  Hence $j(|U|-k+1)$ colors must be chosen.  Thus $L$-colorability
requires $j(|U|-k+1)\le |U|$, which is precisely the inequality
$|U|\le \FR{j(k-1)}{j-1}$.
\end{proof}

The sharpness examples in Proposition~\ref{smallcup} are very dense and have
small cycles.  The special case $j=2$ states that a bipartite graph is
$L$-colorable when $L$ is a $k$-list assignment with
$|\bigcup_{v\in V(G)}L(v)|\le 2k-2$.  This condition forces any two lists to
have at least two common elements, so Proposition~\ref{twocommon} is stronger
than Proposition~\ref{smallcup} for the case $j=2$.  Nevertheless, we show next
that Proposition~\ref{smallcup} remains sharp when $j=2$ even for sparse graphs
with large girth having just one extra edge beyond where Theorem~\ref{choose}
applies.

\begin{theorem}\label{list2k-}
Fix $k,g\in \NN$ with $g$ even and $k\ge2$.  There is a bipartite graph $H_k$
and a $k$-list assignment $L$ on $H_k$ such that $H_k$ is not $L$-colorable,
even though $|\bigcup_{v\in V(H_k)}L(v)|=2k-1$ and $H_k$ has girth at least
$g$ with each proper subgraph having average degree at most $2(k-1)$.
\end{theorem}
\begin{proof}
We use induction on $k$.  For $k=2$, let $H_2$ be $G_2$, the graph consisting
of two $g$-cycles sharing one vertex $u$.  Set $L(u)=\{1,2\}$.  On one cycle,
use lists $\{1,3\}$ and $\{1,2\}$ on the neighbors of $u$ and $\{2,3\}$ on the
rest of the cycle.  Since the number of copies of $\{2,3\}$ is odd, color $1$
must be chosen on a neighbor of $u$.  Interchanging $1$ and $2$ yields the
lists on the other cycle, forcing a neighbor of $u$ to have color $2$.  Now $u$
cannot be colored.  The union of the lists has three colors.

For $k\geq3$, let $r=|V(H_{k-1})|-1$, and let $a+1$ be the number of vertices
of $H_{k-1}$ in the partite set containing the root; note that $a<r$.  We
construct $H_k$ with a list assignment $L$.  Consider a reduced
$(k,(r-1)k,2g)$-graph with underlying tree $T$ and corresponding proper
$[k]$-edge-coloring of $T$.  The root of $T$ will be the root of $H_k$.

For each leaf $v$ of $T$, proceed as follows.  Let $P$ be the full path to $v$
in $T$.  Since $v$ has more than $(a-1)k$ augmenting edges, by the pigeonhole
principle there are $a$ such edges for which the edge along $P$ descending from
the neighbor of $v$ has the same color; call it $c$.  Move the other endpoints
of all $(r-1)k-a$ other augmenting edges at $v$ one step closer to $v$ along
$P$, as in the proof of Theorem~\ref{thm:choice}.  Since $(r-1)k-a>(r-a-1)k$,
by the pigeonhole principle there are $r-a$ of these remaining edges for which
the edge along $P$ descending from the neighbor of $v$ has the same color;
call it $c'$.  Discard all augmenting edges not chosen in these two steps.
After doing this for each leaf $v$ of $T$, the result is a reduced
$(k,r,g)$-graph except for not being bipartite.
 
For each leaf $v$ of $T$, add a copy $H'_v$ of $H_{k-1}$, merging its root with
$v$ and letting each non-root vertex inherit one of the augmenting edges at
$v$, with the vertices in the part opposite $v$ inheriting the $r-a$ edges
whose other endpoints were moved closer to $v$.  Let $H_k$ be the resulting
graph; it is bipartite, and the density bound for its subgraphs is computed as
for $G_k$ in Theorem~\ref{thm:choice}.  Arguing as for $G_k$ also shows that
$H_k$ has girth at least $g$.

Next we produce the list assignment $L$.  Assign list $[k]$ to each non-leaf
vertex of $T$.  By the induction hypothesis, for each leaf $v$ of $T$ there is
a $(k-1)$-list assignment $L'_v$ on $H'$ whose lists are contained in a
$(2k-3)$-set.  For this $(2k-3)$-set use $[2k-1]-\{c,c'\}$, discarding any
additional color if $c'=c$.  Also, let $c_v$ be the color of the edge reaching
$v$ in $T$.  Since $2k-3>k-1$ when $k>2$, we may permute the colors within
$L'_v$ to ensure that $L'$ does not assign color $c_v$ to $v$.

To define lists, form $L(v)$ by adding $c_v$ to the list given by $L'_v$ to the
root.  For $w\in V(H_{k-1})$ other than the root, let $w_v$ be the copy of $w$
in $H'_v$.  Set $L(w_v)=L'_v(w)\cup\{c\}$ if $w$ is in the same partite set as
the root of $H_{k-1}$, and otherwise set $L(w_v)=L'_v(w)\cup\{c'\}$.

It remains to show that $H_k$ is not $L$-colorable.  Let $f$ be a proper
coloring chosen from $L$.  Since the list on each non-leaf vertex of $T$ is
$[k]$ and the coloring is proper, there is a unique $f$-path $Q$ leading to a
particular leaf $v$.  Since the color of each non-leaf vertex on $Q$ agrees
with the color on the edge descending from it along $Q$, the color added to the
list of each vertex $w_v$ in the copy of $H_{k-1}$ at $v$ has been used on its
neighbor in $T$ and is now forbidden from use on $w_v$.  Finding an $L$-coloring
of $H_k$ thus requires finding an $L'$-coloring of $H_{k-1}$, which does not
exist.
\end{proof}

\section{The height of the trees in Theorem~\ref{thm:raugtrees}}\label{sec:height}
The underlying trees in our construction of $(d,r,g)$-graphs are astoundingly
tall; their height in terms of g is a version of the Ackermann function.
Here we show that even for 
$r=1$ and $d=2$, they must be very tall.  In the
discussion below all logarithms are in base $2$.

\begin{theorem}\label{thm:heightbound}
If $G$ is a $(2,1,g)$-graph with height $m$, then $g\le (4+o(1))\log(\log^* m)$.
\end{theorem}
\begin{proof}
For simplicity, we omit floor and ceiling signs; they are not crucial.

For $g\in\mathbb{N}$, let $q=2^{g/4-2}$.  Let $k_{-1}=-1$, $k_0=g-1$, and for
$0\leq i<r$ set
\[ k_{i+1} = 2^{(k_i-g/2+4)/2}+k_i. \]
This yields $g\approx4\log(\log^*k_q)$.  
Let $G$ be a 1-augmented binary tree of height $m$, and let $g$ be the least
integer such that $k_q\ge m$.  We will find in $G$ a cycle of length at most
$g$.

Define integer intervals $I_0,\ldots,I_q$ by $I_j=[m-k_j, m-k_{j-1}-1]$
(deleting any negative elements).  These intervals group the levels in $T$.
The number of levels in $I_j$ is at most $k_j-k_{j-1}$, the value of which is
roughly a tower of height $j$.  However, since we only choose $g$ so that
$k_q\ge m$, the least $j$ with $k_j\ge m$ may be less than $q$, so the
intervals toward the end of the list may be empty.

Let the \emph{mate} of a leaf of $T$ be the other endpoint of its augmenting
edge in $G$.  Let the \emph{type} of the leaf be $j$ if the level of its mate
lies in $I_j$.  We may assume that no leaf has type $0$, since otherwise $G$
has a cycle of length at most $g$.  With each leaf having type in the integer
interval $[1,q]$, some type is assigned to at least $1/q$ of the leaves of $G$.
Fix such a type $t$. 

By averaging, for some vertex $u$ at level $m-k_{t-1}-1$ at least $1/q$ of the
leaves under $u$ have type $t$.  Let $C$ denote the set of all leaves of type
$t$ under $u$.  Let $v$ be the ancestor of $u$ at level $m-k_t$ (or level $0$
if $m<k_t$).  For each leaf $x\in C$, the mate of $x$ is on the $u,v$-path $P$ in
$T$.  Note that $|V(P)|\leq k_t-k_{t-1}=2^{(k_{t-1}-g/2+4)/2}$.

The vertex $u$ has $2^{k_{t-1}-g/4+2}$ descendants at level $m-(g/4-1)$; call
this set $D$.  The subtree rooted at any $y\in D$ has $2q$ leaves.  Call $y$
\emph{full} if at least two leaves of $T$ under $y$ belong to $C$.  Let
$\beta|D|$ be the number of full vertices in $D$.  The number of leaves under
$u$ is $2q|D|$.  Allowing all leaves under full vertices of $D$ and at most one
leaf under non-full vertices, the number of leaves in $C$ under $u$ is at most
$(2q\beta+1)|D|$.  The fraction of leaves under $u$ in $C$ is thus at most
$\beta+\frac{1}{2q}$, but by the choice of $u$ it is at least $1/q$.  Thus
$\beta\geq\frac{1}{2q}$.

Hence at least $2^{k_{t-1}-g/2+3}$ vertices of $D$ are full.  Under each full
vertex of $D$ some two leaves $v$ and $v'$ have mates in $P$.  If $v$ and $v'$
have the same mate $x$, then $x$ completes a cycle of length at most
$2+2(g/4-1)<g$ with the path joining $v$ and $v'$ in $T$.  Otherwise, each full
vertex of $D$ has two leaves under it whose mates are distinct vertices of $P$.
Since the number of full vertices of $D$ exceeds $\binom{|V(P)|}{2}$, by the
pigeonhole principle some two vertices $y, y' \in D$ yield the same pair
$x,x'\in V(P)$ of mates of two leaves under them.  The paths joining those
leaves in the subtrees under $y$ and $y'$ and the edges from those leaves to
$x$ and $x'$ form a cycle of length at most $2(g/4-1)+2(g/4-1)+4$, which equals
$g$.
\end{proof}

\medskip
{\bf Acknowledgment.} The authors thank a referee for helpful comments.

\end{document}